\documentclass[11pt]{amsart}
\usepackage{graphicx,color}
\newtheorem{theorem}{Theorem}[section]
\newtheorem{lemma}[theorem]{Lemma}
\newtheorem{proposition}[theorem]{Proposition}
\newtheorem*{proposition*}{Proposition}
\newtheorem*{theorem*}{Theorem}
\newtheorem*{lemma*}{Lemma}
\theoremstyle{definition}
\newtheorem{definition}[theorem]{Definition}
\newtheorem{example}[theorem]{Example}

\theoremstyle{remark}
\newtheorem{remark}[theorem]{Remark}
\usepackage[english]{babel}
\usepackage[T1]{fontenc}
\usepackage{ifthen}
\usepackage{mathrsfs}
\usepackage{psfrag}
\numberwithin{equation}{section}

\begin{document}
\title{PBNF-transform as a formulation of Propositional Calculus, II}
\author{Pelle Brooke Borgeke}
\address{Linn\oe us university}
\curraddr{}
\email{pelle.borgeke@lnu.se}
\subjclass[2010]{Primary}
\date {9 Februay 2026}
\begin{abstract}Here we show, in the second paper in a series of articles, methods to calculate propositional statements with algebraic polynomials as \emph{symbols} for the connectives, which here are named \emph{operators}. In the first article, we explained this formulation of the Propositional Calculus. In short, we transform  to a dual space, which we here refer to as a polynomial family, which is another shape of DBNF. We name the polynomial families as PBNF, which stands for \emph{Polynomial Boolean Normal Form}. We just use the one law of inference, the rule of Substitution. We can use different polynomial families in the \emph{House} of PBNF, depending on the statement form, making it even simpler. It is also possible to find new theorems and generalize older ones, for example, those given by Church and Barkley Rosser (see follow-up article) concerning duality. 
\end{abstract}
\maketitle
\section{Introduction}
We shall here use Polynomial Boolean Normal Forms, short PBNF, for a mathematical formulation of the \emph{classical} Propositional Calculus (with just 2 truth values, 1 or 0), which we call the PBNF-transform. This is the second article in a series. The method is based solely on Boolean polynomial algebra modulo 2 and uses only one rule of inference, the Rule of Substitution. We consider here the arguments for the transformation to the polynomial families as a \emph{Space of Operators}, $\mathcal{OP}$, which consists of classes of operators, null-class $(p,q$ operands or letters), singular och binary operators and compound statements, which we name forms $(S_{c(k,l}))$, where $c$ stands for classical, $k$ the number of connectives, and $l$ the number of letters. The PBNF transformation could be done in different families in the \emph{House} of PBNF to fit the argument so that we can calculate, in a trivial fashion, the statement with algebraic methods.  In the \emph{protocol}(rules and conditions) below, we outline how it is done in the quadratic case. ($[\cdot]$:=exists, |:=XOR)  \begin{align}[f][\mathcal{H}][a_j][\text{T|F}][p_j,q_k] \bigr(f:\mathcal{OP}\to (\mathbb{Z}_2 | \mathcal{H}\in \mathbb{Z}_2 ),f(op_{\boldsymbol{x}}):= g_{(p_j,q_k,a_0)}(op_{\boldsymbol{x}})\\= a_1pq+a_2p+a_3q+a_4 =(1|0, \big{|}g(p_j,q_k,a_0));
p|q+a =(1|0, \big{|}g(p_j|q_k,a);\\ \mapsto (1|0=\text{(T|F)|(F|T)}), \\ p_j,q_k\in \mathcal{S}=(p,q,p',q',\Leftrightarrow, |;  \bigl). \end{align}
We observe that the range of $f$ is real valued $(1\land 0) \in \mathbb{Z}_2$ and also in $g \in \mathcal{H} \in \mathbb{Z}_2$. If we do not reach a tautology (1) or contradiction(0), we get a polynomial $(a_1pq+a_2p+a_3q+a_4)(0|1)$ which can be evaluated to give either 0 or 1. We can, of course, use the polynomials to pull back to the operator space, applying the inverse $g^{-1}((a_1pq+a_2p+a_3q+a_4))$ as $g$ is a bijection for linear and quadratic polynomial. \footnote{For a cubic polynomial, we have a weak bijection for a couple of interesting ternary operators. Se first article [3].} 
The material is divided as follows: In Section 2, we give a brief look-back at the first article and prove the three other standard binary PBNF families. In the next section 3, we study the generators for the polynomials and compare them with the primitive connectives in standard mathematical logic. Section 4 concentrates on the singular operators, of course using PBNF, with some results in special topics, such as the demi negation [6,11] where we show a way to make a canonical decomposition of the negation. We also study more advanced theories of logic.

\section{Short Resume, and Proof of three more binary PBNF-families}
We start with a summary in eight Paragrafs of our formulation PBNF so far. We also add some common statements for Boolean algebra, to collect this one place.  Therafter we introduce and prove the other three polynomial families.
\textbf{1§.}  We call $x_1,x_2 \ldots , x_n, n\in \mathbb{N} $ for Boolean expressions, with values $x_n \in \mathbb{Z}_{0,1}$. Also $\left[{x_1 x_2} \atop {x_3 x_4} \right],\left[{x_1 x_2} \atop {x_3 x_4}  \right] \ldots$, and $(x_1x_2x_3x_4), (x_1x_2x_3x_4), \ldots$ are Boolean expressions, with Boolean entries/units.
\textbf{2§.}  The Boolean Polynomial Normal Form - PBNF is a canonical transformation of DBNF into families of Boolean polynomials $g(p_j,q_k,a_0) \in \mathcal{H}$ of the form $g(p_i,q_j,a_0)=a_1pq+ a_2p+a_3q+a_4; a_{0 \leq k \geq 4}=0|1$. They can also be represented by $\left[{g_1 \ g_2} \atop {g_3  \ g_4}  \right].$
\textbf{3§.} If $g_j$, $g_k \ldots$, are Boolean polynomials then $g_j g_k\ldots$ and $g_j + g_k + \ldots$ are Boolean polynomials.
\textbf{4§.} If $g$ is a Boolean polynomial then $g'=g+1$ is  a boolean polynomial.
\textbf{5§.} A multiplication and an addition are defined that follow the rules of $\mathbb{Z}_2$ so that $g^2=g$ (the collapsing equation) and $g+g=0$ (the annihilating equation). We have ,for example $$(g_1+g_2)^2=g_1^2+2g_1g_2+g_2^2 \Leftrightarrow g_1+g_2=g_1+2g_1g_2+g_2$$ $$ \Leftrightarrow 2g_1g_2=g_1g_2+g_1g_2=0.$$
Also, a Boolean matrix multiplication can be used to define the polynomials needed in PBNF from $(p,q,1)$ ,for example
$$\left[{p \ p} \atop {q \ 1}  \right] \cdot \left[{q \ p} \atop {p  \  1}  \right]= \left[{p(q+1) \ 0} \atop {(p+q) \ (pq+1)}  \right]  \mapsto (p(q+1),0,p+q,pq+1) \hookrightarrow (\not \Rightarrow,0,|\uparrow) $$
$\textbf{6§.}$ The polynomials are not uniquely written as $g=-g$  for example the linear polynomial $g=p+1=p-1$, because if we add 1 on both sides $p+1+1=p-1+1 \Leftrightarrow  p=p.$ This has no effect on PBNF as the changes in writing could be corrected, and here we do not use the minus sign in the calculus.
$\textbf{7§.}$  The mapping from the range(R) of operator space $\mathcal{OP}$ (that is the boolean expression $x_n$ in $\boldsymbol {1.}$ ) to $\mathcal{H}$ $$g:R(\mathcal{OP}) \rightarrow \mathcal{H}$$ form a bijection from the  $x_n, n=1,2,\ldots,n=0/1$.
The mapping $f$ in the singular case $(j=1)$ is defined by \begin{equation} f(x_1,x_2)= (0, p, p+1,1) \end{equation} in a boolean group $\mathcal{B}_1= \langle 0, p, p+1,1;+ \rangle.$ In the binary case, we have a Boolean polynomial matrix ring \begin{align} \mathcal{B}_2=\langle 0,\left[{0 0} \atop {0 0} \right]p, \left[{1 1} \atop {0 0}  \right]q , \left[{1 0} \atop {1 0}  \right]p+1,\left[{0 0} \atop {1 1}  \right] q+1,\left[{0 1} \atop {0 1}  \right] p+q, p+q+1,\left[{1 0} \atop {0 1}  \right]\\  (p+1)q,(p+1)q+1, p(q+1), p(q+1)+1,\\  (p+1)(q+1),(p+1)(q+1)+1, pq,\left[{1 0} \atop {0 0}  \right] pq+1, 1;+, \cdot,\rangle. \end{align} The mapping $g^{-1}:\mathcal{H} \longrightarrow \mathcal{OP} $ is the $\emph{fiber}$ defined by the inverse image of $g(\boldsymbol {x}).$   
$\textbf{7§.}$ We also proved the first family, which we called the Normal family, $\mathcal{H}^*$. We do not use $\mathcal{OP}^*$ because it refers to the dual in the operator space.
 \begin{tabular}{l l l}
$N(op_{\boldsymbol{x}})$ & Number string in $g(p,q,1)$ & Symbol $\in \mathcal{OP}^* $\\
\hline
$(1) N(\wedge)=1000$ &$(1100)(1010)=(1000)$& $pq$ \\
$(2_p)N(p,\neg p) $&$(1100),(0011)$ &$p,p+1$\\
$(2_q)N(q,\neg q) $&$(1010),(0101)$ &$q,q+1$\\
$(3) N(\vee)=1110$&$1100+1010+1000=1110$ &$(p+1)(q+1)+1$ \\
$(4) N(\Rightarrow)=1011$ &$1000+1100+1111=1011$&$p(q+1)+1$\\
$(5) N(\Leftarrow)=1101$ &$1000+1010+1111=1101$&$(p+1)q+1$\\
$(6) N(\Leftrightarrow) =1001$ &$(1011)(1101)=1001$ &$p+q+1$ \\
$(7) N(\uparrow)= 0111$ &$(1100)(1010)+1111=0111$ &$pq+1$. \\
$(8) N(\downarrow)=0001$ & $1110+1111=0001$ &$(p+1)(q+1)$. \\
$(9) N(|)=0110$ & $1100+1010=0110$ &$p+q$. \\
\end{tabular}  We notice the use of the building blocks $(0, p, q, p+1, q+1,1)$ that use only + and $\cdot$ in $\mathbb{Z}_2$ to fill the table. By adding 1, you negate the operator $ +1 \leftrightarrow \neg op_x$, by adding $p, q$ and $1$ you add component wise $p+q+1=(1,1,0,0) +(1,0,1,0)+(1,1,1,1)=(0,1,1,0)+(1,1,1,1)=(1,0,0,1)$ with $1+1=0$ which gives $ \Leftrightarrow, $ or you can use multiplication as above.
Now we collect some common attribute for Boolean algebra (Here $\mapsto^*$ is transforming to $g(p,q,1))$ \begin{lemma}
1. $\emph{Identity}$ for $\lor:$
$$(p \lor 0=p) \mapsto^* (p+1)(0+1)+1=p, \quad g^{-1}(p)=p$$
for $\land$: $$ (p \land 1 =p) \mapsto^* pq=p1=p $$
2. $\emph{Annihilator}$ for $\lor$: $$ (p \lor 1=1) \mapsto^* (p+1)(1+1)+1=1 $$
for $\land$: $$ (p \land 0=0) \mapsto^* pq=p0=0 $$
3. $\emph{Idempotence}$  for $\lor$: $$ (p \lor p =p) \mapsto^* (p+1)(p+1)+1=p $$
for $\land$: $$ (p \land p=p) \mapsto^* pp=p $$
4. $\emph{Absorption}$  for $\lor$: $$ (p \lor (q \land p)) \mapsto^* (p+1)(pq+1)+1=pq+p+pq+1+1=p  $$ for $\land$: $$ p \land (q \lor p)=p   \mapsto^*  p[(p+1)(q+1)+1]=p .$$
5. For 1 and 0, we find $$ (1\lor 1=1),  (1\land 1=1), (1\lor 0=1), (1\land 0=0),  (0\lor 0=0), (0\land 0=0) $$ 
We also look at the conditional here
6. $(1 \Rightarrow 0) \mapsto p(q+1)+1 = 0$, but
$(0 \Rightarrow 1) \mapsto p(q+1)+1 = 1$ and $(0 \Rightarrow 0) \mapsto p(q+1)+1 = 1$ so whenever $p=0$  the statement is true. Finally $(1 \Rightarrow 1) \mapsto 1(1+1)+1=1$ \end{lemma} \begin{proof} 
Here we used the first normal family $\mathcal{H}^*$, as there are a mix of $\land, \lor$ nothing is gained by changing families and the statements are here also simpel $S_{c(1,2)}.$ \end{proof} In the first article, we proved statements, which we denote $S_{c(k,l)}$, where $k$ is the number of connectives and $l$ is the number of letters in the statement. \begin{example} We have $S_{c(k,l)}: p\land q= \lnot(p \Rightarrow \lnot q)$ and we transform(PBNF): $$g_{(p,q,1)}(S_{c(k,l)})=(pq= g(\lnot)p(q^S+1)+1=g(\lnot)(p(q+1+1)+1)$$  $$=g(\lnot)pq+1=pq+1+1=pq.$$ \end{example} We also compared statements to see if they gave samme or different polynomial to check for equality.
The polynomial could also be used to eliminate connectives. We have already discussed that the conditional $p\Rightarrow q = \lnot p \lor q,$ but if we start by looking att the polynomial for the conditional $p(q+1)+1$, we just look in the table for something that looks alike, and find $(p+1)(q+1)+1$ for the disjunction and we fix this by adding 1 to $p$.
The first line in the table below is just the example we took above concerning the conjunction. Instead of this calculation, we just look at a polynomial that we can transform. The one in this case is the conditional $p(q+1)+1$. We need two changes, negating $q$, and negating the statement to get $\hookrightarrow(pq)=p \land q.$
\begin{tabular}{l l l}
\hline
Statement and polynomial & Alternative polynomial & Alternative statement \\
\hline
$(1) (p\land)q \mapsto pq $ &$p(q+1)+1$ & $\lnot (p\Rightarrow\lnot q)$    \\
$(1') (p\land)q \mapsto pq $ &$pq+1$ & $\lnot(p\uparrow q)=(p\uparrow q) (p\uparrow q)$    \\
$(1'') (p\land)q \mapsto pq $ &$p(q+1)$ & $( p \not\Rightarrow \lnot q)$   \\ 
$(1''') (p\land)q \mapsto pq $ &$(p+1)q$ & $(\lnot p \not \Leftarrow q)$ \\ 
\end{tabular}
Here, we just tried the conjunction to show how it is done. We will come back to non-conditional both ways, as a type of generator of the connectives. 
Now that we have summarized the basis of the system and recalled some items from the first article, we will expand on this by using the idea of extracting the polynomial from its complement: we add 1 everywhere (the result is called the complement space!).  \begin{theorem}
\begin{tabular}{c r l}
$(op_{\boldsymbol{x}})\in \mathcal{OP}$ & \multicolumn{2}{c}{Symbol} $\in  \mathcal{OP}'g(p,q,1'=0) $\\
\hline
$(1) (p \wedge q)$ & $pq+1$ \\
$(2) (p,\neg p $& $p+1,p$\\
$(3) (p \vee q )$ & $(p+1)(q+1)$ \\
$(4) (p \Rightarrow q)$ & $p(q+1)$\\
$(5) (p\Leftrightarrow q)$ & $p+q$ \\
$(6) (p|q)$ & $pq$. \\
\end{tabular} \end{theorem} The polynomials look easier now; we have to negate statements by adding 1. Observe in the example below that we do not use this list for the sub-statements; rather, it negates the compound statement. \begin{example} We prove one of the laws of De Morgan $p\land q = \neg(\neg p\lor \neg q).$ We write the statement as an equivalence
$$ (p\land q) \Leftrightarrow(\neg(\neg p\lor \neg q))\mapsto_{g(p,q,0)} p^S+q^S=pq+pq=0.$$
Of course we see immediately that LS=RS with $g(p,q,1)$ $$pq=((p+1+1)(q+1+1)+1)+1=pq.$$ The example shows that you could use several roads to the goal. \end{example} We shall now prove the existence of the other polynomial families we use.\begin{theorem}
\begin{tabular}{ c c c c }
$ \mathcal{OP}$ & $ \mathcal{H}^{''} g(p',q',1)$& $ \mathcal{H}^{**}g(p',q',0)$\\
\hline
$ p$ & $p+1(0,0,1,1)$ &$p+1(0,0,1,1)$\\
$ q $  & $q+1(0,1,0,1)$ &$q+1(0,1,0,1)$\\
$ \neg p$ & $p$ &$p+1$\\
$ p \wedge q$ &$(p+1)(q+1)$ &$(p+1)(q+1)+1$\\
$ p \vee q $ & $pq+1$ &$pq$ \\
\hline
$ p \Rightarrow q$ & $(p+1)q+1$&$(p+1)q$\\
$ \lnot p\Rightarrow \lnot q $ & $(p+1)q+1$ &$p(q+1)$  \\
$  p \Leftarrow q $ & $p(q+1)+1$&$p(q+1)$ \\
$ \lnot p \Leftarrow \lnot q$ & $(p+1)q+1$ &$(p+1)q$ \\
\hline
$ p\Leftrightarrow q$ & $(p+q)+1$ &$p+q$\\
$ p\uparrow q$ & $(p+1)(q+1)+1$&$(p+1)(q+1)$ \\
\end{tabular} \end{theorem} \begin{proof} Let us now look upon the $\boldsymbol {x_n}$ change in the expression of a matrix-pair as we introduced in the first article. We can use the matrices as a Boolean expression; most of the binary operators function as well in this shape, but operators that have one or two entries with 1 on the anti-diagonal do not work (anti-diagonal matrices turn to change the order as the transpose). We also use them to visualize certain aspects, such as mirror images and other mappings, as follows.
\begin{equation}\left\{  
\begin{array}{ll}
 \boldsymbol{x}_{*}\mapsto (\begin{smallmatrix}x_1 & x_2\\x_3 & x_4\\ \end{smallmatrix})\\
\boldsymbol{x}_{'}\mapsto (\begin{smallmatrix}x_1' & x_2'\\x_3' & x_4'\\ \end{smallmatrix})
\end{array}\right.
\end{equation} This means that the interpositions are the same, but the character has changed of the entries: $\boldsymbol{x}_{*}$ has been turned to it*s complement $\boldsymbol{x}_{'}.$ If we want to change the interpositions, we must first change the Selectors. The simplest way to do this is to take the complement of the Selectors so we get the $1^{st}$ Complement $(0,0,1,1)$ and the $2^{nd}$ Complement $(0,1,0,1)$ as selectors. This will give us a new matrix pair if we also consider negating as before
\begin{equation}\left\{  
\begin{array}{ll}
 \boldsymbol{x}_{''}\mapsto (\begin{smallmatrix}x_4 & x_3\\x_2 & x_1\\ \end{smallmatrix})\\
\boldsymbol{x}_{**}\mapsto (\begin{smallmatrix}x_4' & x_3'\\x_2' & x_1'\\ \end{smallmatrix})
\end{array}\right.
\end{equation}
We see that we have a mirror map for the selectors, and if we even take the complement, we arrive at the pullback complement family $\mathcal{H}^{**}.$ We get the table by adding one for $p$ and one for $q$ in $g(p',q',1)$ (the pullback family) and one for the expression in $g(p',q',1')$ stands for and as $1' \mapsto 1+1=0$, we even use $g(p',q',0)$, not doing anything with $\neg$ as $((p+1)+1)+1=p+1$. 
For uniqueness, we already have a comment on that. The polynomial families can be written in general normal form; that is, they do not have any reducible parts left, as they satisfy the axioms of the real numbers. For example, $g(p,q,1)(p\lor q)=(p+1)(q+1)+1$ is not written in general normal form as it reduces to $pq+p+q$. The first expression is used to facilitate the calculations. \end{proof} We now give the complete table of the polynomials that can be found also in our first article.
\begin{scriptsize}
\begin{center}
\begin{tabular}{||c | c |c |c |c ||}
\hline
$OP$ & $\mathcal{H}^*g(p,q,1)$&$\mathcal{H}'g(p,q,0)$ & $\mathcal{H}^{''}(p',q',1)$ & $\mathcal{H}^{**}(g(p',q',0)$ \\ [0.5ex]
\hline\hline
$ p $ & $p(1,1,0,0)$ &$p'$ & $p'$& $p$\\ 
\hline
$ p' $ & $p'(0,0,1,1)$ &$p$ & $p$& $p'$\\
\hline
$ q $ & $q(1,0,1,0)$ &$q'$ & $q'$& $q$\\
\hline
$ q' $ & $q'(0,1,0,1)$ &$ q$ & $q$& $q'$\\
\hline
$\neg p $ & $p+1$ &$p$ & $p$& $p+1$\\
\hline
$p\vee q $& $(p+1)(q+1)+1$ &$(p+1)(q+1)$ & $pq+1$& $pq$ \\
\hline
$p\wedge q$ & $pq$ & $pq+1$ & $(p+1)(q+1)$& $(p+1)(q+1)+1$\\
\hline
\hline
$p\Rightarrow q$  & $p(q+1)+1$ & $p(q+1)$ & $(p+1)q+1$ & $(p+1)q$ \\ 
\hline
$p \not\Rightarrow q$  & $p(q+1)$ & $p(q+1)+1$ & $(p+1)q$ & $(p+1)q+1$ \\
\hline
$\lnot p\Rightarrow \lnot q$  & $(p+1)q+1$ & $(p+1)q$ & $p(q+1)+1$ & $p(q+1)$ \\
\hline
$p \Leftarrow q$  & $(p+1)q+1$ & $(p+1)q$ & $p(q+1)+1$ & $p(q+1)$ \\
\hline
$p \not\Leftarrow q$  & $(p+1)q$ & $(p+1)q+1$ & $p(q+1)$ & $p(q+1)+1$ \\
\hline
$ \lnot p \Leftarrow \lnot q$  & $p(q+1)+1$ & $p(q+1)$ & $(p+1)q+1$ & $(p+1)q$ \\
\hline
\hline
$p\downarrow q$ & $(p+1)(q+1)$ & $(p+1)(q+1)+1$ & $pq$& $pq+1$\\
\hline
$p\uparrow q$  & $pq+1$ & $pq$ & $(p+1)(q+1)+1$ & $(p+1)(q+1)$ \\
\hline
$p\Leftrightarrow q$ & $p+q+1 $ & $p+q $ & $p+q+1$ & $p+q $ \\
\hline
$p |q$ & $p+q $ & $p+q+1 $ & $p+q$ & $p+q+1 $ \\
\hline
$\iota_1 $ & $(p+1)pq+1 $ & $(p+1)pq$ & $p(p+1)(q+1)+1$ & $p(p+1)(q+1) $ \\
\hline
$\iota_0$ & $(p+1)pq $ & $(p+1)pq+1$ & $p(p+1)(q+1)+1$ & $p(p+1)(q+1) $ \\ [1ex]
\hline
\end{tabular}
\end{center}
\end{scriptsize} \footnote{Church [4] in exercise 16 gives a similar system which he calls the $\emph{dual}$ of an expression, by the mapping $(1\mapsto 0), (\vee \ \mapsto \ \wedge), (\Rightarrow \ \mapsto \ \not \Leftarrow),(\Leftrightarrow \ \mapsto \ \not \Leftrightarrow ), (\Leftarrow \ \mapsto \ \not \Rightarrow), (\downarrow \ \mapsto\ \uparrow ). $ These mappings are in $\mathcal{OP} \mapsto \mathcal{OP}_{**}$, they use $1^{st} Complement \quad p'(0011)$ and $2^{nd} Complement \quad q'(0101)$ as selectors and then negate the result for example $\lor \rightarrow  (\frac{0011}{0101} \mapsto (0111))+(1111)=(1000) \rightarrow \land,$ is what is done when we unbind the transformation. In Stoll [8 p.178], we find some polynomials, also in an exercise, using $\mathbb{Z}_{0,1}$, which is just the Church dual space, without references. Also, Bergmann [2] uses an arithmetic computation without references, but this is based on the max and min of an expression.}  \section{The Generators for the Polynomial operators} In this section, we will discuss the \emph{primitive} connectives and the \emph{generators} for the operators that we got familiar with.
As we have learned on the way, there is in fact only a need for $\{\wedge,\neg \} \mapsto (pq,1)$ or $\{\lor,\neg \} \mapsto (p+q+pq,1)$ to create all the operators in a logic theory, and we prove that now. \begin{theorem} $\{\wedge,\neg \} \mapsto (pq,1)$ or $\{\lor,\neg \} \mapsto (p+q+pq,1)$  generate all the 16 binary operators in a logic theory. In fact, we can use any $(3-1,1110)$ and $\neg$ or $(1-3,(0001))$ and $\neg$ as generators. \end{theorem} \begin{proof} We are allowed to use the arithmetic operators + and $\times$ (which is a repeated addition) to create logic operators. With $\{\wedge,\neg \}$ we have the polynomial $pq$ and $1$ and $p$ and $q$ as $pq$ is a product of $p$ and $q$, which also are the defining areas (the other areas are given in relation to $p$ and $q$), in the unit square.
Recall that the areas in the unit square are $(1)x\land y, (2)x\land y', (3) x'\land y, (4) x'\land y')\mapsto pq, p(q+1), (p+1)q,(p+1)(q+1).$ 
This means that we are able to construct all the possible areas in the unit square by using $(p,q,pq,1)$, which is $X_2=\{ p,q,pq,1\}$, and $\mathcal{P}(X_2)$ the 16 different polynomials. We can use the fiber defined to construct the operators. The proof for $\{\lor,\neg \} \mapsto (p+q+pq,1)$ is the same, but here we are given $p$ and $q$ from the start. We can use any $(3-1,1110)$ and $\neg$ or $(1-3,(0001))$ and $\neg$ operator to get the same generators, for example for 3-1  ${\Rightarrow(1011),\neg } \mapsto (p(q+1)+1,1)$  or 1-3 ${\not \Rightarrow(0,1,0,0),\neg }\mapsto (p(q+1),1)$ because we can use the same construction as for $\lor$ and $\land $. When it comes to $p,q, \neg p, \neg q, 0$ and $1$ $p$ and $q$ just connect to one area and the other cannot be constructed, with just + and $\times$ (with matrix extension selectors can be used as mono-connectives), $p'$ and $q'$ covers three areas out of four, but the last one cannot be constructed. For the trivial operators 0, no areas can be constructed, and for 1, none of the four areas can be constructed. \end{proof} The choice to use $\land$ or $\lor$ and negations is more traditional, and these operators are convenient to have directly. Still, if we use the $\emph{Sheffers stroke}$ $$op_{(p\uparrow q)\boldsymbol{x} (0,1,1,1)} \in \mathbb{Z}_{(1,0)} \mapsto pq+1$$ ; we just need this. It is, of course, at the expense of readability. Just using $(p\uparrow q)$ will  yields very clunky statements; we prove this now. \begin{theorem}Every mapping on $$(1,0) \times (1,0) \rightarrow {1,0}$$ can be generated by an operator with an odd number of 1 in $op_{\boldsymbol{x}(x_1,x_2,x_3,x_4)}$ that follows the votes of $\uparrow$ and $\downarrow.$ In $g(p,q,1)$ we have $pq+1\mapsto p\uparrow(0111)q(\lnot \land)$ and $g(p,q,1)=(p+1)(q+1)\mapsto p\downarrow(0001)=(\lnot \lor)$, but there is no other operator with this property in an algebraic ring $(+\times)$. \end{theorem} \begin{proof} It is well known that starting with either $(\neg, \lor)$ or $(\neg, \land)$, the other operators can be defined by using these as primitives. Now as $\uparrow=\neg \land$ and $\downarrow =\neg \lor$, which is seen by adding one to $(pq+1)+1=pq$ and one to $((p+1)(q+1))+1$ the result follows. For the uniqunes, we observe that $$ (p\uparrow p)\mapsto pp+1=p+1 \mapsto \neg p$$ $$ (p\downarrow p)\mapsto (p+1)(p+1)=p+1 \mapsto \neg p.$$ So these operators can be used to form a negation. No other operator has this crucial property, which could be found by inspecting them with the test above.  \end{proof} \begin{example} We can also solve the polynomial equation \begin{equation}p \uparrow q = p \downarrow q. \end{equation} \begin{align} pq+1=(p+1)(q+1)\ pq+1=pq+p+q+1 \Leftrightarrow p=q \end{align} which gives $$p \uparrow p = p \downarrow p$$ $$\neg p=\neg p.$$ Here we learned that $\uparrow$ and $\downarrow$ are equal when the statement letters are equal, and as this gives $p+1=p+1\mapsto \neg p$, we proved the last line. To convince us, we look at the mapping $$ \frac {\uparrow}{\downarrow } (\begin{smallmatrix}p & 1 & 1 & 0 & 0\\ p & 1 & 1 & 0 & 0 \\ \end{smallmatrix})\mapsto (\begin{smallmatrix} \neg p  & 0 & 0 & 1 & 1\\ \neg p & 0 & 0 & 1 & 1\\ \end{smallmatrix})$$ when we follow the “pick” of these operators. \end{example} \begin{remark} This phenomenon, in which a binary operator becomes a singular operator when acting between two identical statement letters, as in $(p|p)=\neg p$ and $(p\downarrow p)=\neg p$, is a bit of a chimera. It is clear that the result in $\mathcal{OP}$ is not showing a singular operator, as this is $p+1\in \mathcal{OP}^ \mapsto op_{\boldsymbol{x}(0,0,1,1)}$ one of the usual polynomials. If we look at  the table as
\begin{tabular}{@{ }c@{ }@{ }c | c@{ }@{ }c@{ }@{ }c@{ }@{ }c@{ }@{ }c}
$p$ & $p$ &  ($p\uparrow p$)=$\neg p$    \\
\hline
1 & 1 &   \textcolor{red}{0}  \\
0 & 0 &   \textcolor{red}{1}   \\
\end{tabular} as in Quine [7, p. 46] is not so clear because p as a statmentletter can have the truth values (1,0) and $\neg p$ have the opposite (0,1), but here we talk about statement letters connected by an operator (statement forms) and then $p \star p$ has the following possibilities $(1,0)^2= (1,0)\times (1,0)= {(1,1)(1,0)(0,1)(0,0)}.$ This argument raises the question of the character of the negation: it is amorphous and could be thought of as a singular operator, yet in some cases it behaves as a binary operator. (It is clear that $p \lnot q$ does not make sense, but $p (p'0011) q = p'.$) As $\neg p \mapsto p+1$, this means that we use the four-place string $1100+1111=0011$ just as any binary calculation. And above, we saw that $(p\uparrow p)(0001)=\neg p(01)$, which means that, to the left, we have a binary operation and, to the right, a singular operation, raising questions about their equality. If we use Lebniz's definition of equality $x=y \Leftrightarrow x$ has every property that $y$ has, and $y$ has every property that $x$ has”. This does not point to an equality.  If we use Quine's definition and PBNF we get $$x=y \Leftrightarrow (z)(z\in x \Leftrightarrow z\in y)$$ with $(p|p)\mapsto pp+1=p+1$ and $\neg p\mapsto p+1$. In the symbol set $x=(0,p,p+1,1)$ and $y=(0,p,p+1,0$) so this gives equality for $x$ and the $y$ as $(z)(z=0, z=1, z=p, z=p+1\in x \land y).$ If we use the normal form CBNF, we get left side $(p\uparrow p)=(x\land y)'=x'\lor y'=(x’\land y)\lor (x\land y')\lor (x'\land y')$ which are the areas 234 in the Venn-diagram, but this does not work for $x'$, which is area 2 in the Venn-diagram, consisting of $x=1$ and $\lnot x=2$. We get \begin{theorem} The equality $(p\uparrow p)=\neg p$ is  \emph{undecidable} in the operator space or with boolean normal forms. In the polynomial \emph{House} of PBNF, we can prove the equality, using the set-theoretical definition of equality by Quine, so it is decidable. \end{theorem} \begin{proof} We showed that the truth table gave $(p\uparrow p)=p'(0011)$, this means that we got the complement of $p$, areas 3 and 4. This is not the same as the negation of $p$, which is $\lnot p(01)$ with only two truth values, $x$(area 1) and $\lnot x$(area 2). In PBNF, this problem does not arise, because we instead consider the \emph{equality of polynomials}. \end{proof} \begin{proposition} For the singular operator $\neg$ we have $$\neg p = p'=(p\uparrow p)=(p\downarrow p)=(p\uparrow(q)^T)=(q')^T.$$ \end{proposition} \begin{proof} As all the formulas map to $p+1$ and we have a bijection in the fiber $g^{-1}$, this proves the theorem. \end{proof} Here, we also note the facts. \begin{equation} 0|0=0\downarrow 0 =1 \quad \text{and} \quad 1\uparrow1= 1 \downarrow 1=0 \end{equation} \end{remark}  \begin{proposition}Another possibility of primitive connectives are $$(\Rightarrow, \not \Leftarrow \ \mapsto (p(q+1)+1,(p+1)q)$$ but on the contrary $\Leftrightarrow and \not\Leftrightarrow$ does not work.\end{proposition} \begin{proof} If we add the polynomials, we get $1011+0010=1001$, or $(p(q+1)+1+(p+1)q)=p+q+1$, so we get the generators $p,q,1$, and we know this is enough as a primitive operator.
The second case has $1001+0110=1111$, and by multiplication, we get $\langle1001,0110\rangle=0000$, which gives 1 and 0, so we cannot produce $p$ and $q$. \end{proof} \begin{remark} This is a self-dual system of primitive connectives proposed in [3]: $\mathcal{OP}\ni \Rightarrow - \not \Leftarrow  \in \mathcal{OP}_{*}$ As we can see $$\lnot(A)\lnot (p \Rightarrow q \not= \lnot A_1(p\not \Leftarrow q)\mapsto^{*} p(q+1) \not= (p+1)q.$$  Church then proposes two different negations to get $\lnot A=  \lnot_1 A_1$. From the polynomial, we see that what is needed for 0100 to be 0010 is to add 0110(XOR) so we get the identity $\lnot A=  | A_1$. This is not that bad as $\lnot p(1100) \mapsto p+1 0011$ and here $| \ p(1100)\mapsto \ p+p+q=q (1010)$, so this means moving $p$ to $q$, which is the same as  $$(\begin{smallmatrix} 1 & 1\\0 & 0\\ \end{smallmatrix}) \mapsto (\begin{smallmatrix} 1 & 0\\0 & 1\\ \end{smallmatrix})$$ the transpose of $p$. \end{remark} \section{Boolean Linear Polynomials, Singular Operators and Extensions} Here we study in depth Boolean Linear Polynomials, which represent singular operators, acting on null-class operators. These operators are often considered to be of no value, except for the negation $\lnot$, and thus they are not studied in most places, contrary to the binary operators, which are fairly well-studied. However, we shall see that it may be worth taking a look at the linear case. We comment on [6, 11] on the demi negation, a canonical decomposition of the negation, and more. 
We have here \begin{align}f:(x_1,x_2) \longrightarrow \mathcal{P}(p;q) \\ p: f(x_1,x_2)= a_1p+a_2 ; a_n=0, 1; \  q: f(x_1,x_2)= a_1q+a_2. \end{align} 
We have no mixed terms, but we shall later extend this family by letting both $p$ and $q$ be linear polynomials, so we will accept $p+q$(|,0110) or $p+q+1$(I,1001) in this extension. These operators are binary, but they belong to the family of Selectors, so their qualities fit in well here. We shall also note that the line between singular and binary operators is fuzzy.
The first four polynomials come from the singular operators, $\{0,p,\neg,1\}$ or with the identity element $(1,0)$ first $$op_{x_1,x_2}:\{=(1,0),-(0,0),\neg (0,1),+(1,1)\}$$ and we have the $X_p=\{p,1\}$, and $X_q=\{q,1\}$ where we defined the polynomials $f: (x_1;x_2)\rightarrow \mathcal{P}(p)\{0,p,p+1,1\}$, which gives the space $$\mathcal{P}(p;q)=(X_1^*,f).$$ This space comes naturally from the operator space $(X_1, op_{x_1,x_2})$ where the set is $X_1=\{1,0\}$ and $op_{x_1,x_2}$ defined above. 
Of course, we can also define other (little) families by making the change $1\mapsto 0$ and $0\mapsto 1$, so we get $$op_{x_1,x_2}^*=\{=(0,1),-(1,1),+(0,0),\neg (1,0)\}.$$ Further on we can change the input order $(1,0)\mapsto 0,1)$ we get the order $$op_{x_1,x_2}'=\{=(0,1),-(0,0),+(1,1),\neg (1,0)\}$$ and finally this could be read as 0=true and 1=false. $$op_{x_1,x_2}^{**}=\{=(0,1),-(1,1),+(0,0),\neg (1,0)\}.$$
As there are just two heterogeneous inputs, 1,0 and 0,1, and for each space there are two different ways of reading, we get at most four dual spaces, or three dual spaces, for $op_{x_1,x_2} $. To sum up \begin{theorem} For singular operators, we find the dual spaces $$1.f(p,1) \quad \text{The normal space}$$ $$2.f(p,1)\mapsto f(p,1')\quad \text{The complement space: change reading for 1(T) to 0(T)}$$ $$3. f(p,1)\mapsto f(p',1)\quad \text{The pullback space: change input from $p$ to $p'=p+1$}$$ $$4.f(p',1)\mapsto f(p',0)\quad \text{The pullback comomplement space: change input for $p$ and $1'=0$}.$$ \end{theorem} We can think of $\neg$ or $\neg p \mapsto p+1$ as a switching- or negations operator and $p$ or $=$ for neutral operator, or the identity. We call $-$ \emph{Annihilation} operator, and is often defined by the statement $$p \land \lnot p \mapsto p(p+1)=0.$$ We call $+$ for the \emph{Creation} operator and also often defined by $$p \lor \lnot p \mapsto p(p+1)+1=1.$$ These definitions are binary, so we shall avoid them here.
Instead we make the following mapping only using addition on $p$, as we here consider the Boolean algebra as a group \begin{equation} B_2=\langle  0,p,p+1,1;+\rangle \cong B^T_2=\langle 0,q,q+1,1;+\rangle \end{equation} Here we use the matrix-mapping for $p$ and $q$ \begin{align} p(1100)\mapsto (\begin{smallmatrix} 1 & 1\\ 0 & 0 \\ \end{smallmatrix}) ;\ q(1010)\mapsto (\begin{smallmatrix}   1 & 0 \\ 1 & 0 \\ \end{smallmatrix}) \end{align} so that $p=q^T.$ In the next definition, we make a difference between the statement letter “p” and the polynomial “$p$” to make the display clearer. This is not absolutely necessary, as you can see the difference by noting whether a connective precedes the letter. But if we just have $p$, we assume it is the operator of the $1^{st}$ Selector. \begin{definition} \begin{align}(\text{p})[p](=\text{p}\mapsto (p+0)=p, \ -\text{p}\mapsto (p+p)=0,  \  +\text{p} \mapsto p+(p+1)=1,\\ \lnot \text{p} \mapsto (p+1) ; \ \text{p} \in S_{c(0,1)}, \  p\in f(p,1)) \end{align} We see that we use all the polynomials $\mathcal{P}(p)\{0,p,p+1,1\}$ available in this space. \end{definition} \begin{proposition} All singular operators, here $x$, is its own inverse $$x(x\text{p})=(xx)\text{p}=\text{p}.$$ \end{proposition} \begin{proof} Here we use $\hookrightarrow$ as the reverse operation. $$=(=\text{p})\mapsto (p+0)+0=p \hookrightarrow \text{p} $$ $$-(-\text{p}) \mapsto (p+p)+p=p \hookrightarrow \text{p}, (10+10+10=10)$$ $$+(+\text{p}) \mapsto  (p+(p+1))+(p+1)=p \hookrightarrow \text{p}, (10+01+01)=10$$ $$\neg (\neg \text{p})\mapsto (p+1)+1=p \hookrightarrow \text{p}, (10+11+11=10)$$ 
When we examen our proposed extension (| and I) we get $$|(|p)=|(p+q+p)=q+p+q=p.\ \text {For \ II}(p)=I(p+p+q+1)$$  $$= \text{I}(q+1)=p+q+1+q+1=p$$ so the the proposition holds for the extension. \end{proof} The combinations like $-(+p)$ are here called the composition of operators (usually indicated by $\circ$), but the composition is done by adding the operators as shown above, which means that the two of the same operators cancel and two different operators give the third.
\begin{center}
\begin{tabular}{ c| c | c | c | c |}
$\circ$  & $p$ & $-p$ & $\neg p$& $+p$ \\
\hline
$=$ & $p$ & 0  & $\neg p$ & 1 \\
\hline
$ - $ & 0 & $p$ & 1 & $\neg p$ \\
\hline
$\neg $ & $\neg p$ & 1 & $p$ & 0  \\
\hline
$ + $& 1 & $\neg p$ & 0 & $p$  \\
\hline
\end{tabular} \end{center} a full symmetry in the Klein 4-group, with any two different elements givíng the third, and for two of the same, they are inverses.  \begin{remark} We think of these singular operators as $p$ is a circle in the Venn diagram. $p'$ the complement is the rest in the square, $p \cup p'=1, p\lor \lnot p \mapsto p+p+1=1$. In this case, we may identify $\cup$ and $\lor$, whereas with binary operators, we can have non-empty intersections. When we write $p(1,0)$ we think of the operator with two possible values, true=1 or false=0, and $p(1,0)=p(1)=T$ is true and $p(0,1)=p(0)=F $ means $\lnot p$ so $p$ is false or the opposite in a complement space.
In the table, we notice some unintuitive results, for example: $-(\lnot p)=1$ is a tautology, and $+(\lnot p)=0$ is a contradiction, and this is the result of definition 3.11, because $-(\lnot p)\mapsto p+(p+1)=1$ and $+(\lnot p)\mapsto (p+1)+(p+1)=0.$ The table give for $p=(1,0)$ shows that $-p(1,0)=(2,0)=(0,0)=0$ lowers the first value, while $+(-p)=+(0,0)=(0,1)=\lnot p$ lifts the second value. For $-(+p)=-(1,1)=(0,1)=\lnot p$. So the table is  commutative as we get $+(-p) = -(+p).$ For $+(\lnot p)=+(0,1)=(0,2)=(0,0)=0$ as we are characteristic 2 and finally $-(\lnot p)=-(0,1)=(1,1)$ as $-1=1$ here. \end{remark} We summarize in a theorem and here we write $-p=p_-, \lnot p=p_{\lnot}$ and $+p=p_+$. Moreover, we shall use $2 \times 2$ matrices to make the presentation clearer. \begin{theorem} For the singular operators ${p, p_-, p_{\lnot}, p_+}$, the lifting operator, $p_+$, acts on the second coordinate of $p(1,0)$ by adding $1$, and the lowering operator, $p_-$, acts on the first coordinate by adding $1$. The negation operator $p_{\lnot}$ adds $1$ to both coordinates. The transformation is by transposing the input matrices $q_{2 \times 2}^T=p_{2 \times 2}$ $$(\begin{smallmatrix} p_+ & 1 & 0 \\ p_-& 1 & 0 \\ \end{smallmatrix})= (\begin{smallmatrix}1 & 1 \\ 0 & 0 \\ \end{smallmatrix}). $$
If we instead look at $$(\begin{smallmatrix} p_- & 1 & 0 \\ p_+& 1 & 0 \\ \end{smallmatrix})\mapsto (\begin{smallmatrix} 0 & 0 \\ 1 & 1 \ \end{smallmatrix})$$
the transformation is $q_{2 \times 2}^T \mapsto p_{2 \times 2}+1_{2 \times 2}.$
The negation works as $$\begin{smallmatrix} \lnot p(1,0)= & p(1+1 ;& 0+1) \\  \end{smallmatrix} \begin{smallmatrix} =p(0 & 1) \\  \end{smallmatrix}$$ and the transformation is the usual one $\lnot p \mapsto p+1$. We also have that $$(p_x\circ p_{y }) \not = (x \circ y) \circ p$$ where $x$ and $y$  indicate different operators. We also have that $(p_{x}\circ p_{x })=0$. \end{theorem} \begin{proof} Let $(p,p_-, p_{\lnot},p_+)\mapsto(10),(00),(0,1),(11)$ be the values of the  operators acting on $p$. When we add these values, we get
\begin{center}
\begin{tabular}{ c| c | c | c | c |}
$+$  & $p(10)$ & $p_-(00)$ & $p_{\lnot}(01)$ & $p_+(11)$\\
\hline
$p(10)$ & $00$ & $10$  & $11$ & $01$ \\
\hline
$p_-(00)$ & $10$ & $00$ & $01$ & $11$ \\
\hline
$p_{\lnot}(01)$ & $11$ & $01$ & $00$ & $10$ \\
\hline
$p_+(11)$& $01$ & $11$ & $10$ & $00$.  \\
\hline
\end{tabular} \end{center} This means that the composition of operator values of an argument is different from the value of the composed  operators on the same argument, e.g. \begin{equation}p_+ \circ p_- = 11+00=11=p_+ \not =  (+ \circ -)\circ p=(11+00)+10=01=p_{\lnot}.
\end{equation} The only composition we have in this space is addition. We also see that we, in fact, add a boolean as we get 00 on the main diagonal. \end{proof} We find that we have the following relation between the creation och annilation operators \footnote{In the study of the harmonic oscilator $P=D_x^2+x^2$ we can factorize to $P=p_-\circ p_++1$, where $p_+=D_x+ix$ and $p_-=D_x-ix$ ($D_x=\frac{1}{i}\frac{d}{dx}$) are called creation respective annilitaiton operators, which are used to investigate the eigenvalues for $P$} in the inner product $\langle \cdot, \cdot \rangle$ which of course here is addition. \begin{theorem}We have $\langle p_-,p_x\rangle \circ \langle p_x,p_+ \rangle = 1$ but trivially $\langle p_-,p_+\rangle \circ \langle p_+p_- \rangle =0.$ \end{theorem} \begin{proof} The statement follows from with $p=10$ $$\langle(0,0)+(1,0)\rangle+\langle(1,0)+(1,1)\rangle=(1,0)+(0,1)=(1,1)=1$$ and for other values of $x$ in $p_x$: $p=01$ $$\langle(0,0)+(0,1)\rangle+\langle(0,1)+(1,1)\rangle=(0,1)+(1,0)=(1,1)=1$$ and for $p=00$ $$\langle(0,0)+(0,0)\rangle+\langle(1,1)+(0,0)\rangle=(0,0)+(1,1)=(1,1)=1$$ and $p=11$ $$\langle(0,0)+(1,1)\rangle+\langle(1,1)+(1,1)\rangle=(1,1)+(0,0)=(1,1)=1.$$ For the last statement, we get $$\langle(0,0)+(1,1)\rangle+\langle(1,1)+(0,0)\rangle=(1,1)+(1,1)+1=(1,1)=p_+.$$ \end{proof} We could say that these operators act like adjoints; they both add 1, but on different coordinates. As we see from the first table, $\lnot$ can be decomposed as $(+(-p)) = \lnot p$ or $(-(+p)) = \lnot p$, but not in a canonical way, so we get the following decomposition theorem for negation. \begin{theorem} The negation ($\neg \mapsto x+1$) cannot be decomposed as a canonical iteration $f(f(x))=x+1$ in a normal space, but this is possible in the complement space as $\neg p \mapsto p+1+1=p$  and $f(f(x))=x$ there. If we extend the singular operators with | and I, then the iterated negation in a normal space $f(f(x))=x+1$ can be formulated: $f$ of the negation of $f(x)$ is the negation of $x$, $f(f(x)+1)=x+1$. \end{theorem} \begin{proof} If we make the transition to the dual space, we find for the normal space $\neg p \mapsto p+1. $ And we have no canonical iteration here, according to the table above. For the complement space, we find for $p=10$  that $\neg p= p' \mapsto' (p+1)+1=p$ as we added 1 or $\neg (\neg  (10)=10+11=01)=\neg 01 =10.$ Now we have a canonical decomposition, as in fact $$-(-p)=p \mapsto p+p+p=p.$$ This is also true for the other operators (=), (+), and $\neg$ following the table above, but the first two operators are either neutral(=) or 1, and $\neg (\neg (p))$ cannot be thought of as a decomposition of the negation as it is a self-iteration. 
Note that we can use | or  I, for the same result  $$|| p=\text{II} p\mapsto||p=|1=p+q+p+1=p.$$ If we mix this two operators we get $$I|(p)=\text{I}(p+q+1+p)=\text{I}q+1=p+q+q+1=p+1 \hookrightarrow (p+1)= \lnot p.$$
Of course we se at once that $(p+q)+(p+q+1)+p=p+1$ and  $\hookrightarrow(p+1)=\lnot p$. This does not live up to $f(f(x)=x+1$, but it means, as $f(x)+1$ is the negation of $f(x)$, that we get: \emph{$f$ of the negation of $f(x)$ is the negation of $x$}. \end{proof}
\begin{remark}In BD[10, page 233] the author defines the decomposition of $\neg$ as $\sqrt \neg$ (or \S \ in Hu[5] who started this diskussion) but this is non-existent in a \emph{normal} two valued logic as we have seen, and then it is mowed to the 4-valued logic by Belnap and Dunn. The author then get, with $\sqrt \neg= -$ as we are used to here $+p=— $ odd  number $2n+1, n\geq 1$ and even number $4n= – –\ldots=p$ but $2n, n=2n+1=2,6,10 \ldots$ gives $\lnot p$.
\begin{center}
\begin{tabular}{ c| c | c | c | c |}
$\times$  & = & $-$ & $\neg$& $+p $ \\
\hline
$=$ & $p$ & $-p$  & $\neg p$ & $+p $ \\
\hline
($-p)$ & $-p$ & $\neg p$ & $+p $  & $p$ \\
\hline
$(\neg p)$ & $\neg p$ & $+p $ & $p$ & $-p$   \\
\hline
$(+p)$& $+p $ & $p$ & $-p$ & $\neg p$   \\
\hline
\end{tabular}
\end{center}
This is not the system above; here $–p = \neg p$ so $-(-p) = p$. The operator $-$ now changes the input values in two steps, so that for $p(1,0)$ we get $$-(-(p(1,0))=-p_-(0,0)=p_{\lnot}(0,1)=\lnot p$$ or if we start by the second coordinate$$-(-(p(1,0))=-p_+(1,1)=p_{\lnot}(0,1)=\lnot p$$ which means that it does not negate the previous step.
If we form a table in modular 4 with $$=0(0),- (1),(\lnot) --(2),(+)---(3),(= )----$$ we find
\begin{center}
\begin{tabular}
{ c| c | c | c | c |}

* & 0 & 1 &  2&  3 \\
\hline
0 & 0 & 1 & 2 & 3 \\
\hline
1 & 1 & 2 & 3  & 0 \\
\hline
2 & 2 & 3 & 0 & 1   \\
\hline
3& 3 & 0 & 1 & 2   \\
\hline
\end{tabular}
\end{center} that both 2 and 3 form the demi negation, because $3+3=6=2+2+2=0+2$(mod 4)and 2 is its own inverse, and 1 and 3 are each other’s inverses.
But this means  that the unintuitive, but correct, result we commented on in the previous remark will be present, as $- (\lnot) p=-\lnot(1,0)=-(0,1)=(1,1)=+p$ and $+ \lnot (p)=+\lnot(1,0)=+(0,1)=(1,1)=+p$. We can call such an operator a \emph{partial truth changer}, because it makes a $\emph{half-flip}: 10 \mapsto 00$ in contrast to negation, which is a \emph{total truth changer} or a $\emph{flip}: 10 \mapsto 01$. Of course this means that we have four truth values instead of 2, namely for a normal space we can call them $(11)$ \emph{The Big Truth} $(T)$, (10) \emph {The Little Truth} $(t)$, $(0,1)$ \emph {The Little Lie} $(f)$ and $(0,0)$ \emph{The Big Lie} $(F)$. These names come from the fact that we value (1,0) higher than (0,1) in truth value, as we  think of the first coordinate as a “home” coordinate, and the second as an outside coordinate (like the e-mail message: This message comes from outside the organisation) that we can check less.  Here we can see how it works. $$11$$ $$10 \quad \quad 01$$ $$00$$ The generator $-$, which here adds 1 going clockwise and changes from left to right, so that we can, in fact, change the truth value in two steps. If we use the polynomial, we can get the following: here is a normal space $$g(p,q,1)$$ $$g(p,q',1) \quad \quad g(p',q,1)$$ $$g(p',q',1)$$ this means that the input values(truth values) are negated one at a time
$$(\begin{smallmatrix} p & 1 & 1 & 0 & 0\\  q & 1 & 0 & 1 & 0\\ \end{smallmatrix})$$
$$(\begin{smallmatrix} p & 1 & 1 & 0 & 0\\   q+1 & 0 & 1 & 0 & 1\\ \end{smallmatrix})\quad \quad (\begin{smallmatrix}  p+1 & 0 & 0 & 1 & 1\\  q & 1 & 0 & 1 & 0\\ \end{smallmatrix})$$
$$(\begin{smallmatrix}  p+1 & 0 & 0 & 1 & 1\\   q+1 & 0 & 1 & 0 & 1\\ \end{smallmatrix})$$
Now the generator for the negation is 1, and we move clockwise, changing inputs one by one, and we get the same pattern we saw with 4 truth values. This means that we have \begin{proposition} $g'(g'(g(p,q,1)))=g(p',q',1)$ and $g'((g'(g'(g(p,q,1))))=g(p',q',1')$ are the polynomial iteration with negation, adding 1. \end{proposition} \begin{proof} This is a canonical iteration as it comprises two identical operators; it is not a self-iteration and follows the rules set up above of the way that we change one at a time, starting from the left. $p$ and $q$ are mapped into $p+1\mapsto \lnot p$ and $q+1\mapsto \lnot q$, which is a negation. With three changes we get, when we add 1 to zero, just a change in the order of truth $10\mapsto 01$, so that 0 is now representing the true value, we accent this by $\bar 0.$
$$(\begin{smallmatrix} p & 1 & 1 & 0 & 0\\  q & 1 & 0 & 1 & 0\\ \end{smallmatrix})\searrow $$
$$(\begin{smallmatrix} p & 1 & 1 & \bar 0 & \bar 0\\ q & 1 & \bar 0 & 1 & \bar 0\\ \end{smallmatrix})\nearrow \quad \quad \quad    (\begin{smallmatrix}  p+1 & 0 & 0 & 1 & 1\\  q & 1 & 0 & 1 & 0\\ \end{smallmatrix})\downarrow $$
$$ (\begin{smallmatrix}  p & 1 & 1 & \bar 0 & \bar 0\\   q+1 & \bar 0 & 1 & \bar 0 & 1\\ \end{smallmatrix})\uparrow \quad \quad \quad  (\begin{smallmatrix}  p+1 & 0 & 0 & 1 & 1\\   q+1 & 0 & 1 & 0 & 1\\ \end{smallmatrix})\swarrow $$
$$ \nwarrow  (\begin{smallmatrix}  p+1 & \bar 0 & \bar 0 & 1 & 1\\   q+1 & \bar 0 & 1 & \bar 0 & 1\\ \end{smallmatrix})$$ \end{proof} \end{remark} We end this discussion by pointing out the following \begin{lemma} The $(-)$ singular lowering operator for a type of De Morgan negation
\begin{equation}\left\{  
\begin{array}{ll}
i) -(p\lor  q )=\lnot p \land q = p\not \Leftarrow q\\
ii) -(p\land  q )= p \land \lnot q  = p\not \Rightarrow q
\end{array}\right. \end{equation} are negating the conditionals $(\Rightarrow, \Leftarrow ).$ \end{lemma} \begin{proof} For i) we find $(p+1)(q+1)+1+p= (p+1)q$ on the polynomial side and taking the fiber $g^{-1}(p,q)$ we get $p \not \Leftarrow q$ or $\lnot p \land q.$ With the same maneuver, we get ii).
Although $-$ is not a negation, it surely does its best! \end{proof}  In fact we can now define \begin{definition}  
\begin{equation}\left\{  
\begin{array}{ll}
i) -(p\lor  q ):=\lnot(\lnot p \lor q)= p \land \lnot q \\
ii) -(p\land  q ): = \lnot ( p \lor  \lnot q )= \lnot p \land  q
\end{array}\right.
\end{equation} \end{definition} Moreover, if we study the impact of the singular operators on binary statements, we can learn \begin{theorem}
\begin{center}
\begin{tabular}{ c| c | c | c | c | c | c | c | c |c |}
* & $(p\land q) $& $(p\lor q)$ & $(p\Rightarrow q)$ &$(p\Leftarrow q)$ &$(p\Leftrightarrow  q)$ &$\lnot p$&$\lnot q$& $p\downarrow  q$\\
\hline
$=$ & $p\land q$ & $p \lor q$  & $p\land q $ & $p\land q $&$p\Leftrightarrow  q$ &$\lnot p$&$\lnot q$& $p\downarrow  q$\\
\hline
$-$ & $p \land \lnot q $ & $p\lor q$ & $1$ & $p\Leftarrow q$ &$\lnot p/\lnot q$ & $1$ &$1$&$p \Leftarrow q$ \\
\hline
$\lnot $ & $p\uparrow q$ & $1$ & $p\Rightarrow q$ &$p\Rightarrow q$ &$(p|q)$ &$p$&$q$&$p\lor q$ \\
\hline
$+ $& $p\Rightarrow q$ & $p\Leftarrow q$ &$p\land q$&$p\lor q$ &$p\land q$ &$0$&$0$&$\lnot p\land q$\\
\hline
\end{tabular}
\end{center}
\text{The table continues}
\begin{center}
\begin{tabular}{ c| c | c | c | c | c | c | c | c |c |}
* & $p \uparrow q$& $(p|q)$ & $(p\not\Rightarrow q)$ & $(p \not\Leftarrow q)$& $p$  &$q$& $1$ & $0$\\
\hline
=  & $p \uparrow q$& $(p|q)$ & $(p\not\Rightarrow q)$ & $(p \not\Leftarrow q)$& $p$  &$q$& $1$ & $0$\\
\hline
$-$ & $p \Rightarrow q $ & $q$ & $p\land q$ & $p\lor q$ & $0$&$0$&$\lnot p$&$-p$ \\
\hline
$\lnot $ & $p\Rightarrow q$ & $(p \Leftrightarrow  q)$ & $p\Rightarrow q$ &$p\Leftarrow q$ &$\lnot p$ & $\lnot q$&$0$&$1$ \\
\hline
$+ $& $p \Leftarrow q$& $\lnot q$ &$p\uparrow q$& $ p\downarrow q$ &$1 $&$1$&$p$&$+p$\\
\hline
\end{tabular} \end{center} \end{theorem} \begin{remark} We use the polynomial $f(p,1)$ if it is not obvious that we shall use $f(q,1)$, e.g., when $q$ is alone. \end{remark} In the next paper, we continue the work with the operator theory and Polynomial Boolean Normal Form, PBNF, as a formulation of Propositional Calculus. We study the lattices of binary operators and prove extensions of the Church-Rosser dual theorems.
Now follows the Bibliography.

\end{document}